\documentclass[11pt]{amsart}

\usepackage[colorlinks=false]{hyperref}

\usepackage[english]{babel}
\usepackage{packages}
\usepackage{commands}
\usepackage{cd-short}
\usepackage[backend=biber,style=alphabetic]{biblatex}
\usepackage{xcolor}
\usepackage{booktabs}

\newcommand{\fglaw}{\mathop{\!+\!}\limits}
\newcommand{\Fplus}[1]{\mathbin{\mkern-4mu\mathop{+}\limits_{#1}\mkern-4mu}}
\usepackage[margin=1in]{geometry}
\usepackage[nameinlink]{cleveref}
\begin{document}


\title{Level structures and cyclic power operations on the homology of $\EE_\infty$ spaces}
\author{Catherine Li}
\date{July 17, 2026}
\begin{abstract}
	In this document, we discuss how cyclic power operations for periodic complex bordism can be understood through algebraic geometry, following from the theory of power operations for Morava $E$-theory developed by Ando, Hopkins, and Strickland. Using this perspective, we describe how one computes power operations on the $E$-homology of $BU$ and $BU\times \ZZ$, where $E$ is a 2-periodic even $\EE_\infty$ ring. Then we provide some explicit example computations at $p=2$; in particular, we compute the action of the additive operations $\Delta$ on the indecomposables $\widehat{Q}(E_0(BU))$ for $E=KU,E_2$.
\end{abstract}
\maketitle
\tableofcontents


\section{Introduction} 

Let $R$ be an $\EE_\infty$ ring spectrum. Then there is a theory of power operations acting on $\pi_*R$ arising from a refinement of the $m$-th power map on $R$-cohomology. In the case that $R$ is an algebra over a Morava $E$-theory $E$, work of Ando, Hopkins, and Strickland \cite{ando2004sigmaorientationhinfinitymap} developed an algebro-geometric framework for this theory. In particular, Strickland \cite{strickland1998moravaetheorysymmetricgroups} showed that the target $E^0(B\Sigma_{p^r})/I_{\tr}$ of the total additive power operation map $P_E$ corepresents the scheme $\mathrm{Sub}_r(\GG_E)$ of rank $p^r$ subgroups of $\GG_E$. Hence, one can interpret $P_E$ in terms of isogenies of deformations of formal groups. This perspective allows one to explicitly compute power operations for Morava $E$-theory itself, most notably at height 2, where supersingular elliptic curves give a good source of accessible, explicit models of $\GG_E$. For instance, \cite{rezk2008poweroperationsmoravaetheory} and \cite{Zhu} provide detailed calculations of power operations on height 2 $E$-theory at primes 2 and 3, respectively. Aside from this work, there appear to be few explicit computations of power operations via this algebro-geometric framework. 

The goal of this paper is to make this methodology more explicit and to apply its philosophy to other $\EE_\infty$ rings with interesting algebro-geometric descriptions (e.g. the homologies of certain familiar $\EE_\infty$ spaces). Specifically, we give a description of cyclic power operations on the periodic complex bordism spectrum $MP$ via level structures, and we then apply this framework to perform computations in two settings of interest, described in more detail in \ref{sec:intro-MP} and \ref{sec:intro-EBU} below.

\subsection{Power operations on 2-periodic even $\EE_\infty$ rings}
\label{sec:intro-MP}

Consider the total cyclic power operation on $MP_0\cong MU_*$
\[P:MP_0\to MP^0(BC_p).\]
Quillen \cite{quillen1971elementary} gave a formula for $P(x)$ on classes $x\in MP_0$ in terms of Landweber--Novikov operations, following from the fact that the power operation on the coordinate $x_{MP}$ can be interpreted in terms of Chern classes of particular vector bundles. In the case of the bordism classes $c_n=[\CC P^n]$, Johnson and Noel \cite{Johnson_2010} give explicit expressions for $P(c_n)$ in terms of the other $c_i$'s. Most recently, Gu \cite{gu2022power} carried out explicit computations of power operations on classes in $MP_0$ using Hazewinkel generators and the identities of \cite{Johnson_2010}.

In \Cref{sec:theory}, we explain how to calculate cyclic power operations on classes in $MP_0$ from an algebro-geometric point of view, which originates from the fact that $MP_0$ classifies formal groups equipped with a coordinate, and that the cyclic power operation can be described in terms of level structures on these formal groups. We then see that this approach can be used immediately to see how to compute power operations on the homology of certain $\EE_\infty$ spaces, e.g. on $E[BU\times \ZZ]$ and $E[BU]$, where $E$ is a $2$-periodic even $\EE_\infty$ ring. 

While the perspective of \Cref{sec:theory} is likely known to experts, we are not aware of its appearance in the literature for these rings specifically. In the case of $E[B\mathcal{V}]$, where $E$ is Morava $E$-theory, and $B\mathcal{V}=\coprod BU(n)$, Rezk \cite[Theorem 4.9]{Rez15} gives a description of power operations on elements which arise as characteristic classes of vector bundles, such as the universal function $\gamma_{univ}\in (E[B\mathcal{V}])^0(\CP)$; in spirit, this is a variant of Quillen's computation of the power operation on the coordinate $x_{MP}$. Our approach in \Cref{sec:theory} differs in the sense that we deduce expressions for $MP\otimes MP$, $MP[BU\times \ZZ]$, and $MP[BU]$ by observing that the maps relating these spectra commute with power operations, thus avoiding a rederivation of Quillen's calculation.

Then in \Cref{sec:example}, as a simple example, we compute the $C_2$ power operation on $x_1\in MP_0$ and $b_1\in MP_0(MP)$ in terms of their respective formal group laws, and in the former case, we check our expression against a formula of \cite{Johnson_2010}. 

\subsection{Power operations on $KU[BU]$, $E_2[BU]$, and $E_2[BSU]$}
\label{sec:intro-EBU}

This project began as an attempt to understand how to compute power operations on homology, or specifically, how to compute operations on $KU[BU]$, $E_2[BU]$, and $E_2[BSU]$; these computations comprise \Cref{sec:calc} of this paper. Since the homology of $BU$ has an algebro-geometric meaning (as it classifies pointed maps to $\GG_m$), one might naturally ask for the meaning of the total power operation through this lens. 

Simultaneously, these rings are of general interest in chromatic homotopy theory. For instance, $\EE_\infty$ maps $MU\to E$ correspond to $\EE_\infty$ $E$-algebra maps $E[BU]\to E$, and thus the structure of $E[BU]$ can inform us about $\EE_\infty$ complex orientations of $E$; explicitly,  \cite[Theorem 6.5.1]{balderrama2023algebraictheoriespoweroperations} (see also \cite[Proposition 2.8]{rezk2013power}) gives a spectral sequence of signature
\[E_1^{p,q}=H^{p-q}_{E_*\otimes\TT/E_*}(E_*(BU);E_{*+p})\Rightarrow \pi_q(\CAlg_E^{\textrm{aug}}(E[BU],E)),\]
where $E$ is Morava $E$-theory, and $\TT$ is Rezk's monad $\TT$ (for details on $\TT$, see \cite{rezk2009congruencecriterionpoweroperations}, \cite{rezk2017ringspoweroperationsmorava}). At height 1, the module of indecomposables $\widehat{Q}(E_0(BU))$ is projective over $\Delta$ (the operations which act naturally on indecomposables), but it fails to be projective at height 2 \cite[Remark 6.5.5]{balderrama2023algebraictheoriespoweroperations}. Computing power operations may offer a view into how this structural difference manifests. 

Lastly, in a broader sense: much is already understood about the homology of Eilenberg--MacLane spectra, whose infinite loop spaces can be thought of as being ``spaces at chromatic height 0''. The next step would be to examine the homology of spaces at height $1$: $BU$, $BU\times \ZZ$, $BSU$, etc., which are the simplest examples of infinite loop spaces of height 1 objects. 

In \Cref{sec:calc}, we compute the action of $\Delta$ on  $\widehat{Q}(KU_0(BU))$, $\widehat{Q}((E_2)_0(BU))$, and $\widehat{Q}((E_2)_0(BSU))$ for $p=2$. We display our results below, letting $P$ denote the additive $C_2$ power operation.

\begin{thm*}[\autoref{thm:kubu}]
	Let $KU$ denote complex $K$-theory, and let $(KU[BU])_0 \cong
	\ZZ[b_1, b_2, \ldots]$. Let $\theta$ be defined by $P(x) = x^2 + 2\theta(x)$. Then
	\[\theta(b_n)=\sum_{i\geq n}(-1)^n 2^{i-n-1}\left[{i\choose n}+{i-1\choose n-1}\right]b_{n+i} \mod (b_1, b_2, \ldots)^2.\]
\end{thm*}
We obtain similar integral formulas for $E_2[BU]$, which we provide in \autoref{thm:E2bu}. For readability purposes, below we display the mod $(2,a)$ reduction of our results:
\begin{thm*}[\autoref{thm:E2bu}, mod $(2,a)$ version]
	Let $E_2$ be the Morava $E$-theory at height $2$, $p=2$ constructed by Rezk in \cite{rezk2008poweroperationsmoravaetheory}. Let $(E_2[BU])_0 \cong \ZZ_2[[a]][b_1, b_2, \ldots]$. Write $P(x)=Q_0(x)+Q_1(x)d+Q_2(x)d^2$ in Rezk's coordinate $d$ on $E_2^0(BC_2)/I_{\tr}$, with $Q_0(x)= x^2+2\theta(x)$. Then modulo $(2,a)$ and $(b_1, b_2,\ldots)^2$, we find that $\theta(b_i)$, $Q_1(b_i)$ and $Q_2(b_i)$ depend only on $i\bmod 4$:
	\begin{table}[h]
		\centering
		\renewcommand{\arraystretch}{0.9}
		\setlength{\tabcolsep}{10pt}
		\begin{tabular}{c|lll}
			\toprule
			$i \bmod 4$ & $\theta(b_i)$ & $Q_1(b_i)$ & $Q_2(b_i)$ \\
			\midrule
			$0$ & $b_{2i} + b_{2i+3}$ & $b_{2i+1}$ & $b_{2i-1} + b_{2i+2}$ \\
			$1$ & $b_{2i} + b_{2i+3}$ & $b_{2i+1}$ & $b_{2i-1}$ \\
			$2$ & $b_{2i}$            & $b_{2i+1}$ & $b_{2i-1} + b_{2i+2}$ \\
			$3$ & $b_{2i}$            & $b_{2i+1}$ & $b_{2i-1}$ \\
			\bottomrule
		\end{tabular}
		\captionsetup{position=below,aboveskip=5pt,belowskip=-5pt}
		\caption{Power operations on $E_2[BU]$ mod $(2,a)$, $(b_1,b_2,\ldots)^2$.}
	\end{table}
\end{thm*}
Using the canonical map $BSU\to BU$, we immediately determine closed formulas for the case of $E_2[BSU]$ (see \autoref{cor:E2bsu}). In particular, if we let $M=\widehat{Q}((E_2)_0(BSU))$ and $\varepsilon:\Delta\to \FF_2$ be the augmentation which kills $\theta$, $Q_1$, $Q_2$, and $(2,a)$, then we find that $\FF_2 \otimes_\Delta M\cong \FF_2\{d_2\}\oplus \FF_2\{d_3\}$ (cf. \autoref{rmk:admissiblebasis}). This seems to suggest that $M$ is finitely generated as a $\Delta$-module. One could further hope that $E_2[BSU]$ is $K(2)$-locally finite-celled as an $\EE_\infty$ $E_2$-algebra---analogous to $E_n[\ZZ]$ for all $n$, and in contrast to $KU[BSU]$---and that these are all instances of a single structural phenomenon (see \autoref{rmk:topologicallift} for a more detailed discussion). We plan to study $E_2[BSU]$ further in future work. 

To our knowledge, work on power operations on homology is less developed than the corresponding literature on cohomology. One reason for this discrepancy is that power operations on $E[X]$ depend on the $\EE_\infty$-space structure of $X$ itself, whereas this is not the case for $E^X$. For instance, if $X=BU$, then the splitting principle relates the computation of operations on $E^0(BU)\cong E_0[[c_1,c_2,\ldots]]$ to the computation of operations on $E^0(\CP)\cong E_0[[x]]$, and the latter is the content of the aforementioned calculation by Quillen \cite{quillen1971elementary} in the universal case $E=MP$. 

However, in the height 1 case, there have been some related computations of power operations on homology, performed by methods distinct from ours. In \cite[Corollary 6.3.6]{snaith2006dyer}, Snaith computed power operations on $KU_0(BU\times \ZZ;\ZZ/p)$ mod decomposables by directly analyzing the $H_\infty$ structure of $BU\times \ZZ$. For $KU_0(BU)\cong KU_0[b_1,b_2,\ldots]$, Laures~\cite[Theorem C]{laures2003splitting} gives a formula for a power series with coefficients in $\psi(b_i)$'s via tom Dieck operations; it closely resembles the one we obtain in \autoref{cor:mpbu}, specialized to the multiplicative formal group. From Laures' formula, Reeker~\cite[Corollary 3.1--3.3]{reeker2009k1localsubordism}
computes formulas for $\theta(b_r)$ mod $2$, $4$, and $16$, which agree with \autoref{thm:kubu}.

\subsection*{Acknowledgements}
I would like to thank my advisor Allen Yuan for his inspiring guidance and encouragement, and for suggesting this problem and many of the ideas which appear in this paper. I am also grateful to Sanath Devalapurkar, Yuqin Kewang, Akhil Mathew, and Eunice Sukarto for enlightening conversations and comments. This work was supported by NSF grant no. 2140001.

\section{Preliminaries} 
\subsection{Power operations} 
Let's begin by clarifying which $\EE_\infty$ ring structure we consider on $MP$, and then let's recall the construction of the total power operation. 

Let $J$ denote the complex $J$-homomorphism $BU\times\ZZ\to \mathrm{Pic}(\SS)$, or its restriction to the subspace $BU\subseteq BU\times \ZZ$ corresponding to vector spaces of virtual rank 0. The multiplicativity of the construction $V\mapsto S^V$ induces an $\EE_\infty$ structure on $J$, giving rise to $\EE_\infty$ Thom spectra:
\begin{defn}[\cite{may1977e}, \cite{lewis1986equivariant}, \cite{Antol_n_Camarena_2018}]
	Define the complex bordism spectrum $MU$ and the periodic complex bordism spectrum $MP$ as the $\EE_\infty$ Thom spectra
	\[MU = \colim(BU\rarr^J \mathrm{Pic}(\SS)\hookrightarrow \Sp),
	\qquad
	MP = \colim(BU\times\ZZ\rarr^J \mathrm{Pic}(\SS)\hookrightarrow \Sp).
	\]
	There is a canonical element $u\in \pi_2(MP)$ such that $MP_* \cong MU_*[u^\pm]$. 
\end{defn}

Now let $R$ be an $\EE_\infty$ ring spectrum, and let $X$ be a space. Let $x\in R^0(X)$ be a class, which is represented by a map $x:\Susp X\to R$. Then using the diagonal map $X\to X^{\times m}$ and the multiplication map $R^{\otimes m}\to R$, we can obtain the composite 
\[(\Susp X)_{h\Sigma_m}\to \Susp X^{\times m}_{h\Sigma_m}\to R^{\otimes m}_{h\Sigma_m}\to R\]
which represents a class $P_{\Sigma_m}(x)\in R^0(X\times B\Sigma_m)$. This process produces the \textit{total power operation}
\[P_{\Sigma_m}: R^0(X)\to R^0(X\times B\Sigma_m).\]
Restriction along $BC_p\hookrightarrow B\Sigma_p$ induces the \textit{$p$-cyclic \textup{or} $C_p$ power operation} \[P_{C_p}:R^0(X)\to R^0(X\times BC_p).\] These operations are multiplicative but not additive. For the algebro-geometric picture we require a ring homomorphism, which we now construct. The stable transfer $\Sigma^\infty_+ BC_p\to \SS$ induces a map $\tr:R^0(X)\to R^0(X\times BC_p)$, whose image generates the ideal $I_{\tr}\subseteq R^0(X\times BC_p)$. The failure of $P_{C_p}$ to be additive lies in this ideal, and so we obtain a ring homomorphism from the composite
\[P_R^X:R^0(X)\rarr^{P_{C_p}}R^0(X\times BC_p)\to R^0(X\times BC_p)/I_{\tr}\]
which we call the \textit{additive cyclic power operation}. This is the operation we will discuss for the remainder of the paper. When $R$ is clear from context, or $X$ is a point, we suppress it from the notation $P_R^X$. (For more details on power operations, see \cite{rezk2024lectures}.)

\subsection{Level structures} 
In this subsection, we define what we mean by a ``level structure'' in the context of this document. Recall the following construction:
\begin{defn}
	Let $A$ be a finite abelian group, and let $\GG$ be a formal group over a formal scheme $S$. Then we have the following functor from formal schemes to groups:
	\[\underline{\hom}(A,\GG)(T)=\{\mathrm{pairs}\; (i,\ell)\ST i:T\to S,\; \ell\in\hom(A,i^*\GG(T))\}.\]
\end{defn}
If $A=C_n$ and $\GG$ is a formal group over a ring $k$ with coordinate $x$, then $\underline{\hom}(C_n,\GG)$ is corepresentable, i.e.
\[\underline{\hom}(C_n,\GG)=\Spf(k[[z]]/[n](z))\]
where $z=x(\ell(1))$ and $1\in C_n$ denotes a chosen generator of $C_n$ (see \cite[Example 9.3]{ando2004sigmaorientationhinfinitymap}). 
\begin{defn}
	Let $\GG$ be a formal group over a ring $k$ with coordinate $x$. Then a \textit{level-$C_p$ structure} on $\GG$ is a point 
	\[\ell:C_p\to i^*\GG\]
	of $\underline{\hom}(C_p,\GG)$ such that 
	\[\langle p\rangle (z)=0,\]
	where $z=x(\ell(1))$ and $\langle p\rangle (x)$ is the unique power series with $x\cdot \langle p\rangle (x)=[p](x)$. This defines a subfunctor of $\ihom(C_p,\GG)$:
	\[\underline{\mathrm{level}}(C_p,\GG)=\Spf(k[[z]]/\langle p\rangle(z)).\]
	\label{def:level}
\end{defn} 
If $R$ is a 2-periodic even $\EE_\infty$ ring with formal group $\GG_R$ and $p$ is not a zero divisor in $R_0$, then we have an isomorphism of formal schemes over $\Spf(R_0)$
\[\underline{\mathrm{level}}(C_p,\GG_R)=\Spf(R^0(BC_p)/I_{\tr}),\]
where $I_{\tr}$ again denotes the ideal generated by the image of $\tr:R_0\to R^0(BC_p)$. This directly follows from the calculation that $R^0(BC_p)\cong R_0[[z]]/[p](z)$ and $\tr(1)=\langle p\rangle(z)$ (see \cite[Lemma 5.7 and Remark 6.15]{hopkins2000generalized}). 

\begin{rmk}
	In the case that $\GG$ is of finite height, this definition is equivalent to the one given by Drinfeld in \cite{drinfel1974elliptic} or Ando, Hopkins, and Strickland in \cite{ando2004sigmaorientationhinfinitymap}, as discussed in \cite[Example 9.22]{ando2004sigmaorientationhinfinitymap}. Since we'd like to discuss formal groups that are not necessarily of finite height, and the theory of divisors on (or quotients of) formal groups is not necessary to compute cyclic power operations, we have decided to avoid a more meaningful description of level structures. We hope to one day understand to what extent the subgroup picture for Morava $E$-theory can be extended to $MP$. 
\end{rmk}


\section{Operations on homology} 
\label{sec:theory}

In this section, we describe the algebro-geometric interpretations of the additive cyclic power operation $P$ on $MP\otimes X$, where $X=MP$, $\Susp (BU\times \ZZ)$, $\Susp BU$. Then we mention the analogous results for $E\otimes X$, where $E$ is any 2-periodic even $\EE_\infty$ ring. 

\subsection{The case of $MP$} 

Let $R=MP$. Recall that $MP_0\cong MU_*\cong\ZZ[x_1,x_2,\ldots]$ classifies formal group laws. Let us consider the additive cyclic power operation on $MP_0$:
\[P:MP_0\to MP^0(BC_p)/I_{\tr}.\]
After taking $\Spf$, we see that the power operation map involves the following data:
\[ \setlength\arraycolsep{0pt}
P:\begin{array}[t]{ccc}
	\Spec(MP_0) & \larr[3em] & \Spf(MP^0(BC_p)/I_\tr) \\[8pt]
	\left\{\begin{array}{c}
		\text{formal groups}\\
		\text{with a}\\
		\text{coordinate}
	\end{array}\right\}
	&
	\larr[2em]
	&
	\left\{\begin{array}{c}
		\text{formal groups with}\\
		\text{a coordinate and a}\\
		\text{level-$C_p$ structure}
	\end{array}\right\}
\end{array}
\]
To be more precise, let $\GG$ be the universal formal group over $MP_0$, and let 
\[i:MP_0\to MP^0(BC_p)/I_{\tr}\] 
denote inclusion of coefficients. Then let us consider the value of $P$ on  $(MP^0(BC_p)/I_{\tr})$-points. The universal point given by $\id: MP^0(BC_p)/I_{\tr}\to MP^0(BC_p)/I_{\tr}$ consists of the data
\[(\GG, x,\ell:C_p\to i^*\GG)\] 
which is sent to the point $(\GG',x')$, where $x'$ is some coordinate on $\GG'=P^*\GG$:
\[(\GG',x')\lmaps (\GG, x, \ell).\]
In order to give the coordinate $x'$ a meaning, we need to describe it in terms of $x$. Explicitly:
\begin{theorem}
	$P_{MP}$ behaves as follows on $\Spf$: 
	\[\setlength\arraycolsep{0pt}
	\begin{array}{ccc}
		\left(P^*\GG, \prod\limits_{a\in C_p}(x\fglaw_{\GG}x(\ell(a)))\right) & \lmaps^P & (\GG, x, \ell).
	\end{array}
	\]
	\label{thm:mpcase}
\end{theorem}
\begin{rmk}
	It is necessary to clarify our notation, particularly since this picture will continue to appear throughout the remainder of this paper. Strictly speaking, when we say that $x'=\prod(x +_{\GG}x(\ell(a)))$ is a coordinate on $P^*\GG$, we mean that there exists a homomorphism of formal groups $i^*\GG\to P^*\GG$ over $\Spf R^0(BC_p)/I_{\tr}$, and on coordinates, the map
	\[\mathcal{O}_{P^*\GG}=(R^0(BC_p)/I_{\tr})[[x']]\to (R^0(BC_p)/I_{\tr})[[x]]=\mathcal{O}_{i^*\GG}\]
	is given by sending $x'$ to $\prod(x+_{\GG}x(\ell(a)))\in (R^0(BC_p)/I_{\tr})[[x]]$. Furthermore, we use the symbol $x$ for the coordinates of both $\GG$ and $i^*\GG$. In the future, when we discuss other data classified by $R_0$, the notation is analogous. 
\end{rmk}
\begin{rmk}
	The algebro-geometric content of \Cref{thm:mpcase} is essentially a recollection of \cite[Sections 3.1--2]{ando2004sigmaorientationhinfinitymap}, but we do not require that $\GG_R$ is of finite height or that $R_0$ is a complete local ring with perfect residue field of characteristic $p$. This allows us to apply the level structures picture to $MP$ (and later to the homology of $BU\times \ZZ$). We also phrase this result as an explicit formula for the coordinate $x'$ on $P^*\GG$, so that it is directly suited to concrete calculation.
\end{rmk}
\begin{proof}[Proof of \autoref{thm:mpcase}]
	The following argument basically consists of writing down the idea: ``If $x$ is the coordinate, then $P^\CP(x)$ tells you what $P$ does with coordinates.'' 
	
	There is a natural diagram given by
	\begin{cd}
		MP^0(\CC P^\infty \times BC_p)/I_{\tr} & MP^0(\CC P^\infty) \\
		MP^0(BC_p)/I_{\tr} & MP_0
		\arr{1-2}{1-1}_{P^\CP}
		\arr{2-2}{2-1}_{P}
		\arr{2-2}{1-2}
		\arr{2-1}{1-1}
	\end{cd}
	which gives the following diagram after taking $\Spf$:
	\begin{cd}
		i^* \GG & \GG \\
		\Spf(MP^0(BC_p)/I_{\tr}) & \Spec MP_0.
		\arr{1-1}{1-2}^{P^\CP}
		\arr{2-1}{2-2}^{P}
		\arr{1-2}{2-2}
		\arr{1-1}{2-1}
	\end{cd}
	This diagram fits into the pullback diagram defining $P^*\GG$, giving us the unique morphism $i^*\GG\to P^*\GG$:
	\begin{cd}[column sep=1em, row sep=0.7em]
		i^* \GG & & \\
		& P^*\GG & \GG \\
		& \Spf(MP^0(BC_p)/I_{\tr}) & \Spec MP_0
		\arr[bend left=20, inner sep = 0]{1-1}{2-3}_{P^\CP}
		\arr[bend right]{1-1}{3-2}
		\arr[dashed]{1-1}{2-2}
		\arr{2-2}{2-3}
		\arr{2-2}{3-2}
		\arr[phantom, very near start]{2-2}{3-3}^\lrcorner
		\arr{3-2}{3-3}_{P}
		\arr{2-3}{3-3}
	\end{cd}
	Back in rings, the above diagram looks like the following:
	\begin{cd}[column sep= 0.7em, row sep=0.7em]
		MP^0(\CC P^\infty \times BC_p)/I_{\tr} & & \\
		& MP^0(\CC P^\infty)\underset{MP}{\hat\otimes}^P MP^0(BC_p)/I_{\tr} & MP^0(\CC P^\infty) \\
		& MP^0(BC_p)/I_{\tr} & MP_0
		\arr[bend right=15]{2-3}{1-1}^{P^\CP}
		\arr[bend left]{3-2}{1-1}
		\arr[dashed]{2-2}{1-1}
		\arr{2-3}{2-2}
		\arr{3-2}{2-2}
		\arr{3-3}{3-2}^{P}
		\arr{3-3}{2-3}
	\end{cd}
	where the symbol $\hat\otimes^P_{MP}$ in the pushout denotes that the $MP_0$ action on $MP^0(BC_p)$ is through $P$. Then if $x$ is the coordinate on $\GG$, the coordinates-level picture given by the top triangle in the above diagram looks like:
	\begin{cd}[column sep=7em]
		(MP^0(BC_p)/I_{\tr})[[x]]& \\
		(MP^0(BC_p)/I_{\tr})[[x']] & MP_0[[x]].
		\arr{2-2}{1-1}[sloped]^{P^\CP(x)\mapsfrom x}
		\arr[dashed]{2-1}{1-1}
		\arr{2-2}{2-1}^{x'\mapsfrom x}
	\end{cd}
	The dashed arrow therefore gives the morphism $i^*\GG\to P^*\GG$ in coordinates; i.e. it tells us that $x'=P^\CP(x)$. In order to finish the proof, we therefore require one external input: Quillen's classical computation (\cite{quillen1971elementary}, see also \cite[Lemma 2.5.6]{peterson2019formal}) that 
	\[P^\CP(x)=\prod\limits_{a\in C_p}(x\Fplus{MP}[a]_{MP}(z))\]
	where $z=x(\ell(1))$. But this is precisely what we wanted to show, since $[a]_{MP}(x(\ell(1)))=x(\ell(a))$.
\end{proof}

\subsubsection{$MP\otimes MP$} 
Now let $R=MP\otimes MP$, and let $\GG=\GG_R$ be the formal group over $R_0$ induced by the left unit map $\eta_L:MP\to MP\otimes MP$. Let's consider the maps $i,P$ given by the inclusion of coefficients and the power operation, respectively:
\[i,P:R_0\to R^0(BC_p)/I_{\tr}.\]
As before, let's see what $P$ means on $\Spf$. Recall that $MP_0(MP)$ classifies (not necessarily strict) isomorphisms of formal group laws, i.e. it corepresents $\mathrm{Iso}(F_{univ},F_{univ})$. In particular, we let $R_0\cong MP_0[b_0^\pm,b_1,\ldots]$, such that the universal (non-strict) isomorphism of formal group laws $\varphi$ is given by
\[\varphi(x)=b_0x+b_0b_1x^2+b_0b_2x^3+\ldots.\]
Hence, in terms of the universal point, the power operation involves the following data:
\[
P:\begin{array}[t]{ccc}
	\Spec(R_0) & \larr & \Spf(R^0(BC_p)/I_\tr) \\[8pt]
	\left(\GG', x',y'\right) & \lmaps & (\GG, x, y,\ell)
\end{array}
\]
where $\ell:C_p\to i^*\GG$ is a level structure, $x,y$ are coordinates on $\GG$, and $x',y'$ are coordinates on $\GG'$; in particular,  $y=\varphi(x)$ and $y'=\varphi'(x')$, where $\varphi'=P^*\varphi$. Then we have the following:
\begin{theorem}
	$P_{MP\otimes MP}$ behaves as follows on $\Spf$: 
	\[\left(P^*\GG, \prod\limits_{a\in C_p}(x\fglaw_{\GG}x(\ell(a))), \prod\limits_{a\in C_p}\varphi(x\fglaw_{\GG}x(\ell(a)))\right) \lmaps^P (\GG, x, \varphi(x), \ell).
	\]
	\label{thm:mpmp}
\end{theorem}
\begin{proof}
	Consider the left and right unit maps
	\[\eta_L,\eta_R:MP\to MP\otimes MP\]
	which are $\EE_\infty$ and therefore commute with power operations. This fact can be combined with the argument of the previous section to prove the claim. In particular, consider the following commutative diagram:
	\begin{cd}
		MP_0  & MP_0MP & MP_0\\
		MP^0(BC_p)/I_{\tr} & (MP\otimes MP)^0(BC_p)/I_{\tr} & MP^0(BC_p)/I_{\tr}
		\arr{1-1}{1-2}^{\eta_L}
		\arr{2-1}{2-2}^{\eta_L}
		\arr{1-3}{1-2}_{\eta_R}
		\arr{2-3}{2-2}_{\eta_R}
		\arr{1-1}{2-1}^{P_{MP}}
		\arr{1-2}{2-2}^{P_{MP\otimes MP}}
		\arr{1-3}{2-3}^{P_{MP}}
	\end{cd}
	Recall that $\eta_L,\eta_R:MP_0\to MP_0MP$ classify forgetful maps to the source and target, respectively; i.e. on $\Spec$, $\eta_L$ sends the data of an isomorphism of formal group laws $\varphi: F\rwe F'$ to $F$, and $\eta_R$ sends $\varphi$ to $F'$.
	 
	 Hence, on points of $\Spf$, the above diagram yields the following:
	\begin{cd}
		(\GG',x')  & (\GG',x',y') & (\GG',y')\\
		(\GG,x,\ell) & (\GG,x,y,\ell) & (\GG,y,\ell)
		\maps{1-2}{1-1}
		\maps{2-2}{2-1}
		\maps{1-2}{1-3}
		\maps{2-2}{2-3}
		\maps{2-1}{1-1}
		\maps{2-2}{1-2}
		\maps{2-3}{1-3}
	\end{cd}
	From \autoref{thm:mpcase}, we know that 
	\[x'=\prod\limits_{a\in C_p}(x\fglaw_{\GG}x(\ell(a))),\quad y'=\prod\limits_{a\in C_p}(y\fglaw_{\GG}y(\ell(a)))\]
	and since $\varphi$ is a homomorphism, and $y=\varphi(x)$, we can write
	\begin{align*}
		y'&=\prod\limits_{a\in C_p}(y\fglaw_{\GG}y(\ell(a)))\\
		&=\prod\limits_{a\in C_p}(\varphi(x)\fglaw_{\GG}\varphi(x(\ell(a))))\\
		&=\prod\limits_{a\in C_p}\varphi(x\fglaw_{\GG}x(\ell(a)))
	\end{align*}
	as desired.
\end{proof}

\subsubsection{$MP[BU\times \ZZ]$} 

Let's now interpret power operations for 
\[R=MP[BU\times \ZZ],\]
where we again let $\GG=\GG_R$ be the formal group over $R_0$ induced by the left unit map. We use the convention that $R_0\cong MP_0[b_0^\pm,b_1,\ldots]$, where the Thom isomorphism \[\alpha:MP[BU\times\ZZ]\rwe MP\otimes MP\] 
sends $b_i\mapsto b_i$. 

Though this exercise is effectively trivial because the Thom isomorphism is an $\EE_\infty$ equivalence, the algebro-geometric role of $MP[BU\times \ZZ]$ is different from that of $MP\otimes MP$, and so we'll go through this process explicitly. 

Recall that $\Spf(R_0)\cong \Map(\GG,\GG_m)$; i.e. $R_0$ classifies unpointed maps $\GG\to \GG_m$, where the universal map to $\GG_m$ is given by 
\[\gamma(x)=b_0+b_0b_1x+b_0b_2x^2+\ldots.\]
Again, let $i:R_0\to R^0(BC_p)/I_\tr$ denote inclusion of coefficients. In terms of the universal point, we have that the power operation involves the following data:
\[
P:\begin{array}[t]{ccc}
	\Spf(R_0) & \larr & \Spf(R^0(BC_p)/I_\tr) \\[8pt]
	\left(\GG', x',\gamma'\right) & \lmaps & (\GG, x, \gamma,\ell)
\end{array}
\]
where $\ell:C_p\to i^*\GG$ is a level structure, and $\gamma,\gamma'$ are unpointed maps $\GG\to \GG_m$ and $\GG'\to \GG_m$, respectively. Hence, the Thom isomorphism gives us the following corollary:
\begin{cor}
	$P_{MP[BU\times \ZZ]}$ behaves as follows on $\Spf$: 
	\[\left(P^*\GG, \prod\limits_{a\in C_p}(x\fglaw_{\GG}x(\ell(a))), \prod\limits_{a\in C_p}\gamma(x\fglaw_{\GG}x(\ell(a)))\right)  \lmaps^P  (\GG, x, \gamma, \ell).
	\]
	\label{cor:mpbuz}
\end{cor}
\begin{proof}
Because the Thom isomorphism 
\[\alpha:MP[BU\times\ZZ]\to MP\otimes MP\] 
is an $\EE_\infty$ map, it commutes with the power operation. Furthermore, on points of $\Spf$, it is given by ``dividing by $x$'', i.e. 
\[\setlength\arraycolsep{0pt}
\begin{array}[t]{ccc}
	\Spf(MP_0(BU\times\ZZ)) & \larr[3em]^{\alpha} & \Spec(MP_0MP) \\
	\left(\frac{\varphi(x)}{x}=b_0+b_0b_1x+\ldots=\gamma(x)\right) & \lmaps & \left(\varphi(x)=b_0x+b_0b_1x^2+\ldots\right)
\end{array}\]
by definition. More explicitly, an $S$-point $f:MP_0MP\to S$ corresponds to the formal group law isomorphism 
\[f^*\varphi(x)=f(b_0)\bigg(x+ \sum_{i\geq 1} f(b_i) x^{i+1}\bigg).\] 
It is mapped to the point $f\circ \alpha$ of $\Map(\GG,\GG_m)$, which is given by 
\[(f\circ\alpha)^*\gamma(x)=f(\alpha(b_0))\bigg(1+\sum_{i\geq 1} f(\alpha(b_i))x^i\bigg)=f(b_0)\bigg(1+\sum_{i\geq 1} f(b_i)x^i\bigg)=\frac{f^*\varphi(x)}{x}.\] 
Hence by applying the Thom isomorphism to \autoref{thm:mpmp}, we obtain the corresponding diagrams
\[
\begin{il-cd}[column sep = 1 em]
	MP_0MP & (MP^{\otimes 2})^0(BC_p)/I_{\tr}\\
	R_0 & R^0(BC_p)/I_{\tr}
	\arr{1-1}{1-2}
	\arr{2-1}{2-2}
	\arr{2-1}{1-1}^\alpha
	\arr{2-2}{1-2}^\alpha
\end{il-cd}
\leftrightsquigarrow
\begin{il-cd}[column sep = 1 em]
	(\GG', x',\varphi'(x')) & (\GG, x,\varphi(x),\ell)\\
	\left(\GG', x',\frac{\varphi'(x')}{x'}\right) & \left(\GG, x,\frac{\varphi(x)}{x},\ell\right)
	\maps{1-2}{1-1}
	\maps{2-2}{2-1}
	\maps{1-1}{2-1}
	\maps{1-2}{2-2}
\end{il-cd}
\]
and so
\[\gamma'(x')=\frac{\varphi'(x')}{x'}=\frac{\prod_{a\in C_p}\varphi(x+_{\GG}x(\ell(a)))}{\prod_{a\in C_p}(x+_{\GG}x(\ell(a)))}=\prod\limits_{a\in C_p}\gamma(x\fglaw_{\GG}x(\ell(a)))\]
as desired. 
\end{proof}

\subsubsection{$MP[BU]$} 
Now let $R=MP[BU]=MP\otimes \Susp BU$, with $\GG=\GG_R$ being the formal group over $R_0$ induced by the left unit map. Again, consider the maps given by the inclusion of coefficients and the power operation:
\[i,P:R_0\to R^0(BC_p)/I_{\tr}.\]
Let's see what $P$ means on $\Spf$. Recall that $\Spf(R_0)\cong\Map_0(\GG,\GG_m)$. Hence,  the power operation involves the following data:
\[
P:\begin{array}[t]{ccc}
	\Spf(R_0) & \larr & \Spf(R^0(BC_p)/I_\tr) \\[8pt]
	\left(\GG', x',h'\right) & \lmaps & (\GG, x, h,\ell)
\end{array}
\]
where $h$ and $h'$ are now pointed maps $\GG\to\GG_m$ and $\GG'\to\GG_m$, respectively, where the universal pointed map to $\GG_m$ is given by 
\[h(x)=1+b_1x+b_2x^2+\ldots.\] 
Then we have the following corollary:
\begin{cor}
	$P_{MP[BU]}$ behaves as follows on $\Spf$: 
	\[\left(P^* \GG, \prod\limits_{a\in C_p}(x\fglaw_{\GG}x(\ell(a))), \prod\limits_{a\in C_p}\tfrac{h(x+_{\GG}x(\ell(a)))}{h(x(\ell(a)))}\right)  \lmaps^P  (\GG, x, h, \ell).
	\]
	\label{cor:mpbu}
\end{cor}
\begin{proof}
	This argument is similar to the one for \autoref{cor:mpbuz}. It follows from the fact that the $\EE_\infty$ ``inclusion into the 0-component'' map
	\[i_0:MP[BU]\to MP[BU\times \ZZ]\]
	has an algebro-geometric interpretation given by ``dividing by the value at identity'' (see \cite[Remark 5.2.4]{peterson2019formal} for more details), i.e.
	\[\setlength\arraycolsep{0pt}
	\begin{array}[t]{ccc}
		\Spf(MP_0(BU)) & \larr[3em]^{i_0} & \Spf(MP_0(BU\times\ZZ)) \\
		\left(\frac{\gamma(x)}{\gamma(0)}=1+b_1x+\ldots=h(x)\right) & \lmaps & \left(\gamma(x)=b_0+b_0b_1x+\ldots\right),
	\end{array}\]
	which therefore gives us the corresponding diagrams
	\[
	\begin{il-cd}[column sep = 0.5 em, cramped]
		MP_0(BU\times \ZZ) & (MP[BU\times \ZZ])^0(BC_p)/I_{\tr}\\
		R_0 & R^0(BC_p)/I_{\tr}
		\arr{1-1}{1-2}
		\arr{2-1}{2-2}
		\arr{2-1}{1-1}^{i_0}
		\arr{2-2}{1-2}^{i_0}
	\end{il-cd}
	\leftrightsquigarrow
	\begin{il-cd}[column sep = 1 em]
		(\GG', x',\gamma') & (\GG, x,\gamma,\ell)\\
		\left(\GG', x',\frac{\gamma'}{\gamma'(0)}\right) & \left(\GG, x,h,\ell\right)
		\maps{1-2}{1-1}
		\maps{2-2}{2-1}
		\maps{1-1}{2-1}
		\maps{1-2}{2-2}
	\end{il-cd}
	\]
	and hence we have that 
	\[h'(x')=\frac{\gamma'(x')}{\gamma'(0)}=\prod\limits_{a\in C_p}\frac{\gamma(x+_{\GG}x(\ell(a)))}{\gamma(x(\ell(a)))}=\prod\limits_{a\in C_p}\frac{h(x+_{\GG}x(\ell(a)))}{h(x(\ell(a)))}.\qedhere\]
\end{proof}

\subsubsection{The non-universal case}
It is important to note that none of our arguments above were special to $MP$. Specifically, the same arguments apply to $R=E\otimes MP$, $E[BU\times \ZZ]$, and $E[BU]$, where $E$ is a 2-periodic even $\EE_\infty$ ring. The only difference is that we no longer write an explicit formula for $x'$, as $P^{\CP}(x)$ must be computed separately (just as we cited Quillen's computation for $MP$). Below, we compile statements of previous results in the non-universal case:
\begin{theorem}
	\label{thm:nu}
	Let $E$ be a 2-periodic even $\EE_\infty$ ring. Let $\GG_E=\Spf(E^0(\CP))=\Spf(E_0[[x]])$ and let $F$ denote the associated formal group law. We consider the power operation $P_R$ in three cases below: $R=E\otimes MP$, $E[BU\times\ZZ]$, $E[BU]$. In each case, let $(\GG,x)$ denote the pullback of $(\GG_E,x)$ along $E\to R$, and let $\GG'=(P_R)^*\GG$ and $x'=P_R^\CP(x)$. 
	\begin{enumerate}
		\item Let $E_0(MP)=E_0[b_0^\pm,b_1,\ldots]$, which corepresents $\mathrm{Iso}(F,F_{univ})$ via the universal isomorphism $F\cong F_{univ}$
		\[\varphi(x)=b_0x+b_0b_1x^2+\ldots.\]
		Then $P_{E\otimes MP}:\Spf(E_0MP)\leftarrow \Spf((E\otimes MP)^0(BC_p)/I_{\tr})$ does the following on points:
		\[\left(\GG', x', \prod\limits_{a\in C_p}\varphi(x\fglaw_{\GG}x(\ell(a)))\right) \lmaps (\GG, x, \varphi(x), \ell).
		\]
		\item Let $E_0(BU\times \ZZ)=E_0[b_0^\pm,b_1,\ldots]$, which corepresents $\Map(\GG,\GG_m)$ via the universal (unpointed) map $\GG\to \GG_m$
		\[\gamma(x)=b_0+b_0b_1x+\ldots.\] 
		Then $P_{E[BU\times \ZZ]}:\Spf(E_0(BU\times \ZZ))\leftarrow \Spf((E[BU\times \ZZ])^0(BC_p)/I_{\tr})$ does the following on points:
		\[\left(\GG', x', \prod\limits_{a\in C_p}\gamma(x\fglaw_{\GG}x(\ell(a)))\right)  \lmaps  (\GG, x, \gamma, \ell).
		\]
		\item Let $E_0(BU)=E_0[b_1,b_2,\ldots]$, which corepresents $\Map_0(\GG,\GG_m)$ via the universal pointed map $\GG\to \GG_m$
		\[h(x)=1+b_1x+\ldots.\] 
		Then $P_{E[BU]}:\Spf(E_0(BU))\leftarrow \Spf((E[BU])^0(BC_p)/I_{\tr})$ does the following on points:
		\[\left(\GG', x', \prod\limits_{a\in C_p}\tfrac{h(x+_{\GG}x(\ell(a)))}{h(x(\ell(a)))}\right)  \lmaps  (\GG, x, h, \ell).
		\]
	\end{enumerate}
\end{theorem}

\section{Example calculations}
\label{sec:example} 
Let's now explicitly demonstrate how one uses these algebro-geometric descriptions to compute power operations on $R$. The general principle is as follows: one starts with some expression $f$ relating to $\GG_R$ (such as a coordinate on $\GG_R$, a map $\GG_R\to \GG_m$, etc.). Then the algebro-geometric picture gives some formula for the pullback $(P_R)^* f$ of $f$ along the power operation $P_R$. By comparing $(P_R)^* f$ to this formula and matching coefficients, we can compute the power operation on the classes that appear in the expression $f$.

\subsection{The $C_2$ power operation on $x_1\in MP_0$} 
We begin with a simple example computation that we can compare to Quillen's formula. Let $MP_0=\ZZ[x_1,x_2,\ldots].$ Let the group law of $\GG=\GG_{MP}$ be given by
\[x+_\GG y=x+y+x_1xy+\ldots.\] 
Our goal is to compute the additive $C_2$ power operation $P(x_1)$ using the description we gave earlier.

Recall that the homomorphism $i^*\GG\to P^*\GG$ is given in coordinates by $x'=x(x+_\GG z)=:\psi(x)$. We may as well use the 2-series for ease of computation; in particular, 
\[[2]_{P^*\GG}(\psi(x))=\psi([2]_\GG(x)).\]
Thus we can line up the coefficients of $x^2$ to find $P(x_1)$ after expanding out the expression above:
\[2x(x+_\GG z)+P(x_1)(x(x+_\GG z))^2+\ldots=(x+_\GG x)(x+_\GG x+_\GG z).\]
Specifically, in collecting $x^2$ terms, we obtain that the above equation becomes the following:
\[2(xz+x^2a_1)+P(x_1)x^2z^2=4x^2a_1 +x_1x^2z+2xz\]
where $a_1\in MP_0[[z]]$ denotes the coefficient of $x$ in $x+_\GG z$. Simplifying, we obtain that
\[P(x_1)=\frac{1}{z^2}(x_1z+2a_1).\]
Let's now compare this with a formula of Johnson--Noel derived from Quillen's computation.

\begin{fact}[Lemma 5.15 of \cite{Johnson_2010} at $p=2$] 
	We begin by explaining some notation. Let $\alpha=(\alpha_0,\alpha_1,\ldots)$ be a multi-index and $\bar{\alpha}=(\alpha_1,\ldots)$. Let $|\alpha|=\sum_{i\geq 0}\alpha_i$. Let $|\alpha|'=\sum_{i\geq 0}i\alpha_i$. If $y_0,y_1,\ldots$ is a list of variables, then let $y^\alpha=y_0^{\alpha_0}y_1^{\alpha_1}\cdots$. Let $\mu(n;\bar{\alpha})$ be the coefficient of $b^{\bar\alpha}$ in $(1+b_1+b_2+\cdots)^n$. Then
	\[z^{2n}P([\CC P^n])=\sum_{|\alpha|=n}a^\alpha s_\alpha[\CC P^n]\]
where $a_i$ is the coefficient of $x^i$ in $x+_\GG z$, and $s_\alpha$ is the $\alpha$-indexed Landweber--Novikov operation. Also, one can compute that 
\[s_\alpha[\CC P^n]=\mu(-(n+1);\bar\alpha)[\CC P^{n-|\alpha|'}].\]
\end{fact}
We're working with the convention that $x_1=-[\CC P^1]$. Since only multi-indices $(1,0,\ldots)$ and $(0,1,0,\ldots)$ give nonzero operations, we obtain that
\[z^2P(-x_1)=a_0s_{(1,0,\ldots)}(-x_1)+a_1s_{(0,1,0,\ldots)}(-x_1)=zs_{(1,0,\ldots)}(-x_1)+a_1 s_{(0,1,0,\ldots)}(-x_1).\]
Then looking at coefficients of $(1+b_1+b_2+\ldots)^{-2}$, we find that $s_{(1,0,\ldots)}(-x_1)=-x_1$ and $s_{(0,1,0,\ldots)}(-x_1)=-2$, and we therefore obtain the desired result: 
\[P(-x_1)=\frac{-1}{z^2}(x_1z+2a_1).\]

\subsection{The $C_2$ power operation on $b_1\in MP_0(MP)$} 
For ease of calculation, let us proceed with the identification 
\[MP_0(MP)\cong MP_0(BU\times \ZZ)\cong MP_0[b_0^\pm,b_1,\ldots].\]
Our goal in this example is to compute $P(b_1)$, where $P$ is the additive $C_2$ power operation. As in \autoref{cor:mpbuz}, let us denote $\gamma(x)=b_0+b_0b_1x+b_0b_2x^2+\ldots,$
and let $x'=x(x+_\GG z)$. We saw that 
\[\gamma'(x')=P(b_0)+P(b_0b_1)x'+P(b_0b_2)(x')^2+\ldots=\gamma(x)\gamma(x+_\GG z).\]
Comparing constant terms, we see that $P(b_0)=b_0\gamma(z)$. Since the power operation is multiplicative, 
\[1+P(b_1)(x(x+_\GG z))+P(b_2)(x(x+_\GG z))^2+\ldots=\frac{\gamma(x)\gamma(x+_\GG z)}{b_0\gamma(z)}.\]
Thus we match the coefficients of $x$ on both sides, and obtain that 
\[P(b_1)=\frac{1}{z}\left(b_1+a_1 \frac{\gamma^{(1)}(z)}{\gamma(z)}\right)\]
where $\gamma^{(1)}(z)$ is the formal derivative of $\gamma(z)$, and $a_1\in (MP\otimes MP)_0[[z]]$ denotes the coefficient of $x$ in $x+_\GG z$.

\section{Results for $KU[BU]$, $E_2[BU]$, and $E_2[BSU]$} 
\label{sec:calc}
In this section, we derive closed formulas for power operations on $KU[BU]$, $E_2[BU]$, and $E_2[BSU]$. In particular, we compute the additive $C_2$ power operation $P$ on $b_i$ mod decomposables for both $KU_0(BU)$ and $(E_2)_0(BU)$, where $E_2$ is Morava $E$-theory of height $2$ at $p=2$ with Lubin--Tate parameter $a$. Results for $E_2[BU]$ then immediately give us formulas for $E_2[BSU]$.

Note that throughout this paper, we use the convention that $\binom{a}{b}=0$ if $a<0$, $b<0$, or $a<b$.

\subsection{Height 1: $KU[BU]$} 

Let us consider the additive $C_2$ power operation $P$ on $KU[BU]$, where $KU$ is complex $K$-theory:
\[P:KU_0[b_1,b_2,\ldots]\to KU_0[b_1,b_2,\ldots][[z]]/(z-2).\]
Let $\theta$ be the operation such that $P(x)=x^2+2\theta(x).$ Then we have the following result:
\begin{theorem} $\theta$ acts on the generators as follows:
	\[\theta(b_n)=\sum_{i\geq n}(-1)^n 2^{i-n-1}\left[{i\choose n}+{i-1\choose n-1}\right]b_{n+i} \mod (b_1,b_2,\ldots)^2.\]
	In particular, if we also take coefficients mod 2, then
	\[\theta(b_n)=b_{2n}+b_{2n+1} \mod 2,(b_1,b_2,\ldots)^2.\]
	\label{thm:kubu}
\end{theorem}
\vspace{-1em}
\begin{proof}
	Let $h(x)=1+b_1x+b_2x^2+\cdots$. Let $\GG$ be the multiplicative formal group with group law given by $x+_\GG z=x+z-xz$, i.e. the coordinate arising from the class $1-L$ where $L$ is the tautological bundle over $\CP$. One can explicitly compute that
	\[P(x)=x(x+_\GG z),\] 
	e.g. by considering the class $(1-L)^{\otimes 2}$ in the $C_2$-equivariant $K$-theory of $\CP$; in fact, the aforementioned group law actually arises from an $\EE_\infty$ orientation on $KU$, e.g. see \cite[Theorem 6.5.3]{balderrama2023algebraictheoriespoweroperations}. Then applying \autoref{thm:nu}, we have that $P^*h(x)=h'(x')$ is given by the following expression:
	\[1+P(b_1)x(x+_\GG z)+P(b_2)(x(x+_\GG z))^2+\ldots=\frac{h(x)h(x+_\GG z)}{h(z)}.\]
	By rearranging terms and matching coefficients, we can therefore derive a formula for $P(b_n)$. 
	
	In particular, let $y:=x+_\GG z$ and $x':=xy=x(x+_\GG z)$. If we reduce mod decomposables, then 
	\[(h(z))^{-1}=1-b_1z-b_2z^2-\ldots.\] 
	Hence we can reduce our equation to the following:
	\[\sum_{n\geq 1}P(b_n)(x')^n=\sum_{j\geq 1}(x^j+y^j-z^j)b_j,\]
	and identifying $b_j$ coefficients, we obtain 
	\[\sum_{n\geq1} [b_j]P(b_n)(x')^n=x^j+y^j-z^j,\]
	where $[b_j]P(b_n)$ denotes the coefficient of $b_j$ in $P(b_n)$. We can now prove the theorem with a simple application of Waring's formula, a combinatorial identity that expresses power sums in terms of elementary symmetric polynomials (see e.g. \cite{gould1999girard}):
	\begin{fact}[Waring's formula]
		\label{fact:waring}
		Let $A$, $B$ be indeterminates. Then for all $j\geq 1$,
		\[A^j+B^j=\sum_{k\geq 0}\left[\binom{j-k}{k}+\binom{j-k-1}{k-1}\right](-AB)^k(A+B)^{j-2k}.\]
	\end{fact}
	Let $A=x$, $B=y$. Mod transfers, $z=2$ and $x+y=2$, so we have the following identity:
	\[2^j+\sum\nolimits_{n\geq1} [b_j]P(b_n)(x')^n =x^j+y^j=\sum_{k\geq 0}\left[\binom{j-k}{k}+\binom{j-k-1}{k-1}\right](x')^k(-1)^k2^{j-2k}.\]
	Now matching coefficients of $(x')^n$, we obtain the following:
	\[P(b_n)=\sum_{j\geq 2n}(-1)^n2^{j-2n}\left[\binom{j-n}{n}+\binom{j-n-1}{n-1}\right]b_j.\]
	Then we reindex by $j-n=i$ and divide by 2 to obtain the desired expression for $\theta$:
	\[\theta(b_n)=\sum_{i\geq n}(-1)^n 2^{i-n-1}\left[{i\choose n}+{i-1\choose n-1}\right]b_{n+i} \mod (b_1,b_2,\ldots)^2.\qedhere\]
\end{proof}

\subsection{Height 2: $E_2[BU]$ and $E_2[BSU]$} 
In this section, we describe the computation of power operations for $E[BU]$ and $E[BSU]$, where $E=E_2$ is a Morava $E$-theory of height 2 at $p=2$. 

In general, the computation of power operations at height 2 is feasible because there exist models of Morava $E$-theory at any prime given by supersingular elliptic curves. Specifically, one can obtain a height 2 formal group by taking the formal completion of a supersingular curve $C$ at the identity. 

Then there is a universal deformation of $C$ (given by a particular family of curves parameterized by a variable $a$); the associated formal group of this deformation yields a model of $\GG_E$ with Lubin--Tate parameter $a$. One can compute the universal order $p$ subgroup $H$ of $C$, and determine the universal isogeny on $C$ with kernel given by $H$; this isogeny allows one to compute the $\Sigma_p$ power operation on $E$ (see \cite{Zhu} for more details).

We proceed by using the elliptic curve model for $E$ given by Rezk in \cite{rezk2008poweroperationsmoravaetheory}: let $E_0=\ZZ_2[[a]]$, and consider the curve $C$ over $\ZZ_2[a]$ defined by projective equation 
\[Y^2Z+aXYZ+YZ^2-X^3=0.\] 
To arrive at the formal group associated to this curve, we consider an affine neighborhood about the identity $[0:1:0]\in\PP^2$, and coordinates $u=X/Y, v=Z/Y$, giving us the equation
\[v^2+auv+v=u^3\]
which has identity at $(u,v)=(0,0)$. The completion of this curve at the identity gives a universal deformation over $E_0$ with coordinate $u$. Rezk finds that the universal order two subgroup $H$ of $C$ is defined over 
\[\ZZ_2[[a]][d]/(d^3-ad-2)\cong E^0(BC_2)/I_{\tr},\] 
with generator $Q$ given by $(u(Q),v(Q))=(d,-d^3)$. Note that $\tr(1)$ is not $d^3-ad-2$, but these rings are isomorphic because they classify the same moduli problem. 

Let $P$ be a point such that $(u(P),v(P))=(u,v)$. By examining the universal isogeny of $C$ with kernel $H$, Rezk computes the power operation on the coordinate, obtaining that 
\[u'=P^\CP_E(u)=-u\cdot u(P+Q)\]
where the right-hand-side can be expanded in terms of the elliptic curve group law to obtain an expression in $u$ and $d$. Since $\ZZ_2[[a]][d]/(d^3-ad-2)$ is free of rank 3 as an $E_0$ module with basis $\{1,d,d^2\}$, the power operation $P_{E[BU]}$ admits a unique decomposition
\[P_{E[BU]}(x)=Q_0(x)+Q_1(x)d+Q_2(x)d^2,\]
defining additive operations $Q_i$; furthermore, Frobenius congruence at the prime 2 implies that $Q_0(x)=x^2 \mod 2$, and so we define the operation $\theta$ by $Q_0(x)=x^2+2\theta(x).$

\begin{theorem}
	\label{thm:E2bu}
	As above, let $E=E_2$ be Morava $E$-theory of height 2 at $p=2$ in Rezk's model, with $E_0=\ZZ_2[[a]]$. Consider the additive $C_2$ power operation on $E[BU]$ given by
	\[P:E_0[b_1,b_2,\ldots]\to E_0[b_1,b_2,\ldots][[d]]/(d^3-ad-2),\]
	and let 
	\[P(x)=Q_0(x)+Q_1(x)d+Q_2(x)d^2,\] 
	where $Q_0(x)=x^2+2\theta(x)$, as above. Then we have the following mod $(b_1, b_2, \ldots)^2$:
	\begin{align}
		\theta(b_n) &= b_{2n} + \sum_{i\geq 1} b_i\Biggl(\sum_{\substack{j, k, \ell\geq 0\\ 3j+2k + 3\ell + 3=n+i}} \left[\binom{i-j}{j}+\binom{i-j-1}{j-1}\right]\binom{i-2j}{n-j}\binom{k+\ell}{\ell}a^k 2^\ell \Biggr), \tag{1}\\
		Q_1(b_n) &= \sum_{i\geq 1} b_i\Biggl(\sum_{\substack{j, k, \ell\geq 0\\ 3j+2k + 3\ell + 1=n+i}} \left[\binom{i-j}{j}+\binom{i-j-1}{j-1}\right]\binom{i-2j}{n-j}\binom{k+\ell}{\ell}a^k 2^\ell \Biggr), \tag{2}\\
		Q_2(b_n) &= \sum_{i\geq 1} b_i\Biggl(\sum_{\substack{j, k, \ell\geq 0\\ 3j+2k + 3\ell + 2=n+i}} \left[\binom{i-j}{j}+\binom{i-j-1}{j-1}\right]\binom{i-2j}{n-j}\binom{k+\ell}{\ell}a^k 2^\ell \Biggr). \tag{3}
	\end{align}
\end{theorem}	
To obtain formulas for $E[BSU]$, we simply use the fact that the generators $d_i$ of $E_0(BSU)\cong E_0[d_2,d_3,\ldots]$ can be chosen such that the map $E_0(BSU)\to E_0(BU)$, which is induced by the canonical inclusion $BSU\to BU$, behaves as follows (see e.g. \cite[Corollary 3.5]{laures2002characteristic}):
\[d_n \mapsto \begin{cases*}
		2\cdot b_n & if $n=2^s$, \\
		b_n & otherwise.
	\end{cases*}\]
Let $B_\theta(i,n)$, $B_{Q_1}(i,n)$, and $B_{Q_2}(i,n)$ denote the large interior sums in the above formulas for $\theta(b_n)$, $Q_1(b_n)$, and $Q_2(b_n)$ respectively, i.e. let
\[\theta(b_n)=b_{2n}+\sum_{i\geq1}B_\theta(i,n)b_i,\qquad
Q_1(b_n)=\sum_{i\geq1}B_{Q_1}(i,n)b_i,\qquad
Q_2(b_n)=\sum_{i\geq1}B_{Q_2}(i,n)b_i.\]
Then we immediately have the following:
\begin{cor}
	\label{cor:E2bsu}
	As above, consider the additive $C_2$ power operation on $E[BSU]$ given by
	\[P:E_0[d_2,d_3,\ldots]\to E_0[d_2,d_3,\ldots][[d]]/(d^3-ad-2),\]
	and let $P(x)=Q_0(x)+Q_1(x)d+Q_2(x)d^2$. Then we have the following mod $(d_2, d_3, \ldots)^2$:
	\begin{align}
		\theta(d_n)&=
		\begin{cases}
			d_{2n}+\sum\limits_{t\geq 1}B_\theta(2^t,n)d_{2^t}+\sum\limits_{\substack{i\geq 2\\ i\neq 2^t}}2B_\theta(i,n)d_i\hspace{15.13pt}
			& \hspace{15.13pt} \text{if } n=2^s\\[2ex]
			d_{2n}+\sum\limits_{\substack{i\geq 2\\ i\neq 2^t}}B_\theta(i,n)d_i+\sum\limits_{2^t>n}\frac{B_\theta(2^t,n)}{2}d_{2^t}
			& \hspace{15.13pt} \text{if } n\neq 2^s
		\end{cases}\\
		Q_1(d_n)&=
		\begin{cases}
			\hphantom{d_{2n}+{}}\sum\limits_{t\geq 1}B_{Q_1}(2^t,n)d_{2^t}+\sum\limits_{\substack{i\geq 2\\ i\neq 2^t}}2B_{Q_1}(i,n)d_i
			& \text{if } n=2^s\\[2ex]
			\hphantom{d_{2n}+{}}\sum\limits_{\substack{i\geq 2\\ i\neq 2^t}}B_{Q_1}(i,n)d_i+\sum\limits_{2^t>n}\frac{B_{Q_1}(2^t,n)}{2}d_{2^t}
			& \text{if } n\neq 2^s
		\end{cases}\\
		Q_2(d_n)&=
		\begin{cases}
			\hphantom{d_{2n}+{}}\sum\limits_{t\geq 1}B_{Q_2}(2^t,n)d_{2^t}+\sum\limits_{\substack{i\geq 2\\ i\neq 2^t}}2B_{Q_2}(i,n)d_i
			& \text{if } n=2^s\\[2ex]
			\hphantom{d_{2n}+{}}\sum\limits_{\substack{i\geq 2\\ i\neq 2^t}}B_{Q_2}(i,n)d_i+\sum\limits_{2^t>n}\frac{B_{Q_2}(2^t,n)}{2}d_{2^t}
			& \text{if } n\neq 2^s
		\end{cases}
	\end{align}
\end{cor}
\begin{rmk}
	\label{rmk:admissiblebasis}
	\cite{rezk2008poweroperationsmoravaetheory} computes the commutation and Adem relations for the operations $\theta$, $Q_1$, and $Q_2$ on $E_0$, and finds that the ring $\Delta$ (consisting of the natural operations acting on indecomposables) has an admissible basis; i.e. it is a free module over $E_0$ on generators of the form
	\[\theta^jQ_{k_1}\cdots Q_{k_r}, \quad j\geq0, r\geq 0, k_i\in\{1,2\}.\]
	In the introduction, we noted that $M=\widehat{Q}((E_2)_0(BSU))$ appears to have a nice structure: if $\varepsilon:\Delta\to \FF_2$ is the augmentation which kills $\theta$, $Q_1$, $Q_2$, and $(2,a)$, then we find that $\FF_2 \otimes_\Delta M\cong \FF_2\{d_2\}\oplus \FF_2\{d_3\}$, suggesting that $M$ may be finitely generated as a $\Delta$-module. Explicitly, from the formulas of \autoref{cor:E2bsu}, it is not too hard to see that the following identities hold mod $(2,a)$:
	\begin{alignat*}{2}
		\theta(d_n)&=d_{2n},\quad Q_1(d_n)=Q_2(d_n)=0 &&\qquad \text{if $n=2^s$,}\\
		\theta(d_n)&=d_{2n}+d_{2n+3} &&\qquad \text{if $n\neq2^s$, $n\equiv0,1 \bmod 4$,}\\
		\theta(d_n)&=d_{2n} &&\qquad \text{if $n\neq2^s$, $n\equiv2,3 \bmod 4$,}\\
		Q_1(d_n)&=d_{2n+1} &&\qquad \text{if $n\neq2^s$,}\\
		Q_2(d_n)&=d_{2n-1} &&\qquad \text{if $n\neq2^s$, $n$ odd,}\\
		Q_2(d_n)&=d_{2n-1}+d_{2n+2} &&\qquad \text{if $n\neq2^s$, $n$ even.}
	\end{alignat*}
	Furthermore, $d_2$ and $d_3$ do not vanish in $\FF_2 \otimes_\Delta M$, because mod $(2,a)$, operations on $d_n$ do not contain $d_i$ terms for $i<n$, and the $a^0$ coefficients of operations on $d_2$ and $d_3$ do not contain $d_2$ and $d_3$ terms. From these observations and the commutation relations of \cite{rezk2008poweroperationsmoravaetheory}, it follows that $\FF_2 \otimes_\Delta M\cong \FF_2\{d_2\}\oplus \FF_2\{d_3\}$. Interestingly, the fact that $Q_1(d_n)=Q_2(d_n)=0$ mod $(2,a)$ when $n=2^s$ suggests that $M$ may split as a module mod 2. 
\end{rmk}
\begin{rmk}
	\label{rmk:topologicallift}
	After finding a finite-celled resolution of $M$, one might hope to lift this algebraic resolution to the level of $K(2)$-local $\EE_\infty$ $E_2$-algebras. Specifically, the above evidence would suggest that $E_2[BSU]$ is $K(2)$-locally finite-celled as an $\EE_\infty$ $E_2$-algebra. This appears analogous to the case of $E_n[\ZZ]$: $E_n[\ZZ]$ is finite-celled for all $n$, which can be seen via the symmetric power of spheres filtration on $\ZZ$ (e.g. see \cite[Section 2.5]{mathew2014thursday}). Note that it is essential that $E_2$ has height at least $2$, because $KU[BSU]$ is not finite-celled; its homology is not finitely generated as a $\TT$-algebra, which can be seen from \autoref{thm:kubu} (the sense is that there are ``not enough operations'' at height 1 to cover all the generators in a finite-length procedure).
\end{rmk}

\begin{proof}[Proof of \autoref{thm:E2bu}]
	In analogy to the case of height 1, we can use Rezk's model and \autoref{thm:nu} to compute $P(b_i)$ for $b_i\in E_0(BU) \cong \ZZ_2[[a]][b_1,\ldots]$. In particular, if we let $h(x)=1+b_1x+b_2x^2+\ldots$, then we obtain the equation
	\[1+P(b_1)u'+P(b_2)(u')^2+\ldots=\frac{h(u)h(u(P+Q))}{h(d)}\]
	where $u'=P^\CP_E(u)=-u\cdot u(P+Q)$, and $d$ denotes the coordinate of the universal 2-torsion point $Q$. We can solve this equation to obtain $P(b_i)$. In particular, from now on, we will work mod decomposables. 
	
	Letting $w=u(P+Q)$, we obtain that 
	\[\sum_{n\geq 1}P(b_n)(u')^n=\sum_{i\geq 1}b_i(u^i+w^i-d^i)\]
	mod decomposables. By matching $b_i$ coefficients, we find that 
	\[\sum_{n\geq 1}[b_i]P(b_n)(u')^n=u^i+w^i-d^i.\]
	
	The bulk of the work of the proof rests on a remarkable fact proven in the thesis of Schumann \cite{schumann2014k}. Schumann shows computationally (via manipulations of the elliptic curve group law) that the power operation on the coordinate is a rational expression:
	\begin{fact}[Remark 3.3.3 of \cite{schumann2014k}] 
		For $w=u(P+Q)$, we have that
		\[w=\frac{d-u}{1+d^2u}.\]
	\end{fact}
	We are not currently aware of a more theoretical way to prove this surprising identity (though we would be very interested in finding one). Using this fact, we find that $u+w=d+d^2u'$, as follows:
	\[u+w=u+\frac{d-u}{1+d^2u}=\frac{d+d^2u^2}{1+d^2u}=\frac{d+(d^3u-d^3u)+d^2u^2}{1+d^2u} =d+d^2\frac{u^2-du}{1+d^2u}=d+d^2u'.\]
	Then we can yet again apply Waring's formula (\autoref{fact:waring}), with $A=u$ and $B=w$:
	\begin{align*}
		d^i+\sum\nolimits_{n\geq1} [b_i]P(b_n)(u')^n &=u^i+w^i\\
		&=\sum_{j\geq 0} \left[\binom{i-j}{j}+\binom{i-j-1}{j-1}\right] (u')^j(d+d^2u')^{i-2j}\\
		&=\sum_{j,m\geq 0} \left[\binom{i-j}{j}+\binom{i-j-1}{j-1}\right] \binom{i-2j}{m}(u')^{j+m}d^{i-2j+m}
	\end{align*}
	where we expanded $(d+d^2u')^{i-2j}$ to obtain the last equality. Thus, after matching coefficients of $(u')^n$, we have that
	\[[b_i]P(b_n)=\sum_{j\geq 0} \left[\binom{i-j}{j}+\binom{i-j-1}{j-1}\right] \binom{i-2j}{n-j}d^{n+i-3j}.\]
	To isolate the operations $Q_0$, $Q_1$, and $Q_2$, we must examine the recursion that occurs in powers of $d$. Since $\{1,d,d^2\}$ is a basis for $E_0[d]/(d^3-ad-2)$ over $E_0$, we write
	\[d^m=f_m+g_md+h_md^2\]
	defining sequences $f_m,g_m,h_m\in E_0$. By repeatedly multiplying by $d$, one sees that $f_m,g_m,h_m$ each satisfy shifts of the same recurrence relation, which we use to define the sequence $D_r$: 
	\[D_r =\begin{cases}
		aD_{r-2}+2D_{r-3}, & \text{if } r\geq 3,\\
		0, & \text{if } r=2 \text{ or } r\leq 0,\\
		1, & \text{if } r=1.
	\end{cases}\]
	Specifically, for $m\geq 1$, we find that $f_m=2D_{m-2}$, $g_m=D_m$, and $h_m=D_{m-1}$, and
	\[d^m= 2D_{m-2}+D_md+D_{m-1}d^2.\]
	Now let us consider the generating function $G(x):=\sum_{m\geq 0}D_m x^m$. We claim that
	\[G(x)=\frac{x}{1-ax^2-2x^3},\]
	which can be seen as follows:
	\begin{align*}
		(1-ax^2-2x^3)G(x) &= \sum_{m\geq 0}D_m x^m-a\sum_{m\geq 2}D_{m-2} x^m-2\sum_{m\geq 3}D_{m-3} x^m\\
		&=D_0+D_1x+D_2x^2-aD_0x^2+\sum_{m\geq 3}(D_m-aD_{m-2}-2D_{m-3})x^m\\
		&= x. 
	\end{align*}
	Thus, we can obtain a closed-form expression for $D_m$:
	\begin{align*}
		D_m &= [x^m]\frac{x}{1-ax^2-2x^3}\\
		&=[x^m]\left(x\sum_{k\geq 0}(ax^2+2x^3)^k\right)\\
		&=[x^m]\left(\sum_{k,\ell\geq 0}\binom{k+\ell}{\ell}a^k 2^\ell x^{2k+3\ell+1}\right)\\
		&=\sum_{\substack{k, \ell\geq 0\\ 2k + 3\ell + 1=m}}\binom{k+\ell}{\ell}a^k 2^\ell.
	\end{align*}
	Now we can apply this to our expression for $P(b_n)$. We have from above that
	\[[b_i]P(b_n)=\sum_{j\geq 0}C_j(i,n)d^{n+i-3j}\]
	where $C_j(i,n)=\left[\binom{i-j}{j}+\binom{i-j-1}{j-1}\right]\binom{i-2j}{n-j}$. In order to expand the above expression using the identity
	\[d^{n+i-3j}=2D_{n+i-3j-2}+D_{n+i-3j}d+D_{n+i-3j-1}d^2,\]
	which is only true for $n+i-3j\geq 1$, we must look at the values of $i, n,j$ for which $C_j(i,n)\neq 0$ and $n+i-3j\leq 0.$
	
	In order for $C_j(i,n)$ to be nonzero, it must be that $i\geq 2j$, $n\geq j$, since we use the convention that $\binom{a}{b}=0$ if at least one of $a,b$ is negative, or if $a<b$. Therefore, it must be that $n+i-3j=0$, and the only term affected by this condition is the $j=n$ term of $[b_{2n}]P(b_n)$. In this case, $C_{n}(2n,n)=2$, and since 
	\[2D_{-2}+D_{0}d+D_{-1}d^2=0,\]
	this boundary case simply contributes one extra $2\cdot b_{2n}$ term to $P(b_n)$. In other words, we obtain that 
	\[P(b_n)= 2\cdot b_{2n} + \sum_{i\geq 1} b_i\Biggl(\sum_{j\geq 0}C_j(i,n)\left(2D_{n+i-3j-2}+D_{n+i-3j}d+D_{n+i-3j-1}d^2\right)\Biggr).\]
	Then we simply substitute the formula 
	\[D_m=\sum_{\substack{k, \ell\geq 0\\ 2k + 3\ell + 1=m}}\binom{k+\ell}{\ell}a^k 2^\ell\]
	from above. We obtain $Q_0$, $Q_1$, and $Q_2$ by isolating the coefficients of $1$, $d$, and $d^2$, respectively, and we obtain $\theta$ by dividing $Q_0$ by $2$, yielding the desired identities:
	\begin{align}
		\theta(b_n) &= b_{2n} + \sum_{i\geq 1} b_i\Biggl(\sum_{\substack{j, k, \ell\geq 0\\ 3j+2k + 3\ell + 3=n+i}} \left[\binom{i-j}{j}+\binom{i-j-1}{j-1}\right]\binom{i-2j}{n-j}\binom{k+\ell}{\ell}a^k 2^\ell \Biggr), \tag{1}\\
		Q_1(b_n) &= \sum_{i\geq 1} b_i\Biggl(\sum_{\substack{j, k, \ell\geq 0\\ 3j+2k + 3\ell + 1=n+i}} \left[\binom{i-j}{j}+\binom{i-j-1}{j-1}\right]\binom{i-2j}{n-j}\binom{k+\ell}{\ell}a^k 2^\ell \Biggr), \tag{2}\\
		Q_2(b_n) &= \sum_{i\geq 1} b_i\Biggl(\sum_{\substack{j, k, \ell\geq 0\\ 3j+2k + 3\ell + 2=n+i}} \left[\binom{i-j}{j}+\binom{i-j-1}{j-1}\right]\binom{i-2j}{n-j}\binom{k+\ell}{\ell}a^k 2^\ell \Biggr). \tag{3}
	\end{align}
\end{proof}

\printbibliography
\end{document}